\documentclass[10pt,reqno]{amsart}
\usepackage[T1]{fontenc}
\usepackage[latin1]{inputenc}
\usepackage{amsmath}
\usepackage{amssymb}
\usepackage{amsthm}
\usepackage{enumerate}
\usepackage{graphics}
\usepackage{tikz}
\usetikzlibrary{patterns}
\usepackage{color}
\usepackage[most]{tcolorbox}
\tcbuselibrary{breakable, skins}
\usepackage{caption}
\usepackage{subfigure}
\usepackage[a4paper,
left=2cm, right=2cm,
top=3cm, bottom=3cm]{geometry}

\usepackage{hyperref}
\hypersetup{
    colorlinks,
    linkcolor={luh-dark-blue},
    citecolor={luh-dark-blue},
    urlcolor={orange}
}

\definecolor{luh-dark-blue}{rgb}{0.0, 0.313, 0.608}

\newtheorem{satz}{Proposition}[section]
\newtheorem{lem}[satz]{Lemma} 
\newtheorem{remark}[satz]{Remark}
\newtheorem{thm}[satz]{Theorem}
\newtheorem*{mainthm}{Main Theorem}

\newcommand{\chookrightarrow}{\mathrel{\lhook\joinrel\relbar\kern-.8ex\joinrel\lhook\joinrel\rightarrow}}

\newcommand{\R}{\mathbb{R}}

\newcommand{\N}{\mathbb{N}}		   
\newcommand{\Z}{\mathbb{Z}}				
\newcommand{\C}{\mathbb{C}}	

\newcommand{\T}{\mathbb{T}}

\newcommand{\e}{\varepsilon}

\normalsize
\DeclareMathOperator{\Ima}{Im}
\DeclareMathOperator{\supp}{supp}
\DeclareMathOperator{\esssup}{ess-sup}

\numberwithin{equation}{section}

	\title[Symmetry of periodic traveling waves]{Symmetry of periodic traveling waves for nonlocal dispersive equations}

\author{Gabriele Bruell}
\address{Institute for Analysis, Karlsruher Institute of Technology (KIT), D-76128 Karlsruhe, Germany}
\email{gabriele.bruell@kit.edu}
\author{Long Pei}
\address{School of Mathematics (Zhuhai), Sun Yat-sen University, 519082 Zhuhai, China}
\email{peilong@mail.sysu.edu.cn}
\thanks{Date: \today }

\begin{document}

\maketitle
\allowdisplaybreaks

\begin{abstract}
Of concern is the  \emph{a priori} symmetry of traveling wave solutions for a general class of nonlocal dispersive equations
\[
u_t + (u^2 +Lu)_x=0,
\]
where $L$ is a Fourier multiplier operator with symbol $m$. Our analysis includes both homogeneous and inhomogeneous symbols. We characterize a class of symbols $m$ guaranteeing that periodic traveling wave solutions are symmetric under a mild assumption on the wave profile. Particularly, instead of considering  waves with a unique crest and trough per period or a monotone structure near troughs as classically imposed in the water wave problem, we formulate a \emph{reflection criterion}, which allows to affirm the symmetry of periodic traveling waves. The reflection criterion weakens the assumption of monotonicity between trough and crest and enables to treat \emph{a priori} solutions with multiple crests of different sizes per period. Moreover, our result not only applies to smooth solutions, but also to traveling waves with a non-smooth structure such as peaks or cusps at a crest. The proof relies on a so-called \emph{touching lemma}, which is related to a strong maximum principle for elliptic operators, and a weak form of the celebrated \emph{method of moving planes}. 

\end{abstract}

\section{Introduction}
The present manuscript is devoted to symmetry of periodic traveling wave solutions for general nonlocal dispersive equations of the form\footnote{We use the notation $\hat f$ to denote the Fourier transformation of a function $f$.}
\begin{equation}\label{eq:Equation}
u_t+(u^2+Lu)_x=0, \qquad \widehat{Lu}(t,k)=m(k)\hat u(t,k),
\end{equation}
where $t>0$  and $x\in \R$ denote the time and space variables, respectively. The linear operator $L$ is a Fourier multiplier operator with real symbol $m$. For certain classes of symbols $m$ and a mild \emph{a priori} assumption on the wave profile, which we call \emph{reflection criterion}, we show that periodic traveling solutions of \eqref{eq:Equation} are symmetric and have exactly one crest per period. Our study includes smooth solutions as well as the so-called \emph{highest waves} exhibiting a cusp or corner singularity at their crests.

\medskip

Our main motivation for studying the symmetric structure for periodic solutions of \eqref{eq:Equation} steams from the full water wave problem governed by the Euler equations in two dimensions. For the water wave problem the existence and symmetric structure of steady, periodic waves in various settings has been subject of intense study during the last decades. Classical existence results for periodic traveling waves assume that the wave profile is symmetric
 \cite{Gerstner1809, Dubreil-Jacotin, Hur2001, Constantin-Strauss2002, Constantin-Strauss2004}.  However, this is not necessarily a restriction as many studies concerned with {a priori} symmetry of traveling waves show. 
The first result in the context of irrotational flows goes back to Garabedian \cite{Garabedian} in 1965, who proved that if every stream line obeys a \emph{monotone profile}, which means that it has a unique maximum and minimum per period, all of them located on a vertical line, respectively, then the periodic steady wave is symmetric. The proof has been simplified in \cite{Toland1998} in the 90s and the authors of \cite{OS} proved that in fact any steady periodic solution for irrotational flows is symmetric under the condition that (only) the surface wave profile is monotone. Under the same condition on the wave profile, it is shown in  \cite{ Constantin-Escher2004b,Constantin-Escher2004} that steady waves are symmetric  in the context of flows with vorticity. In the same setting, the authors of \cite{Matioc2013} replaced the assumption of a monotone wave profile by the condition that all  streamlines achieve their global minimum on the same vertical line and are monotone in a small (one side) neighborhood. Then they prove   that the wave is symmetric and has actually a monotone profile. Notice that {a priori} the wave profile is allowed to have multiple crests per period. 
A further result for waves which may have {a priori} several crests within a period was established in \cite{Hur2007}. Under the assumption that the wave is monotone near a unique global trough in each period and every stream line attains a minimum below this trough, it is shown that the wave is symmetric and has a single crest per period. The above mentioned results concern flows without stagnation. 
In fact, it was shown in \cite{EEW} that there exist rotational flows with critical layers having multiple crests of different sizes within one period.



\medskip

 It naturally raises the question whether a local monotone structure, either near a global trough or between a trough and a crest, is indispensable to guarantee the symmetry of steady solutions in water wave problems. The answer is negative for \emph{solitary water waves} as indicated in \cite{Craig} for irrational flows and \cite{Hs} for waves with vorticity (without stagnation). The precise decay of solitary waves enables the application of the so-called \emph{method of moving planes}, which goes back to Alexandrov \cite{Alk} and was refined by the works of Serrin \cite{Serr}, Gidas, Ni, and Nirenberg \cite{GNN} and many others. Concerning water wave model equations in the regime of shallow water, there are results confirming the symmetry of solitary waves irrespective of the precise decay rate, see for instance \cite{BEP16,Pei-DP} for the Whitham and Degasperis--Procesi equation, respectively, and  \cite{A} for a class of general dispersive equations.
%

\medskip

\medskip

The situation is more complicated for periodic waves in the context of the water wave problem and model equations of the form \eqref{eq:Equation}. Due to the lack of decay at infinity, a monotone structure, either near a global trough or from it to a global crest,  is required in previous results  in order to initiate the method of moving planes. 
We relax the assumption on a monotone structure as far as possible, still guaranteeing that the method of moving planes is directly applicable. To this end, we introduce the following criterion:
\begin{description}\label{cond: reflection criterion} 
\item[Reflection criterion]
A $2\pi$-periodic continuous function  $\phi$ is said to satisfy the \emph{reflection criterion} if  there exists $\lambda_*\in [0,2\pi)$ such that
\[
	\phi(x)>\phi(2\lambda_*-x)\qquad \mbox{for all}\qquad x\in (\lambda_*, \lambda_*+ \pi).
\]
\end{description}
This criterion is weaker than the classical \emph{monotone profile} assumption, and does not impose a monotone structure at any particular point on the wave profile so that it can be used to confirm the symmetry of periodic waves with arbitrarily many crests and troughs per period. 

\medskip 

Classically, the method of moving planes strongly relies on an elliptic maximum principle for local equations. Dealing with a genuinely nonlocal equation like \eqref{eq:Equation}, an in-depth study of the nonlocal operator is required in order to apply the method of moving planes. We impose assumptions on the symbol of the nonlocal operator $L$, which corresponds to the dispersion relation of \eqref{eq:Equation}, guaranteeing that the action of $L$ reflects a weak form of an elliptic maximum principle. We call it the \emph{touching lemma} (cf. Lemma \ref{touching lemma} below).
In our analysis, we impose the following assumption on the symbol $m$: 

\begin{description}
\item[Assumption] The symbol $m$ is even, real and satisfies one of the following conditions:
 \begin{itemize}
 \item[(S)] $m\in S^r(\Z)$ for some $r<0$  is inhomogeneous, and the  sequence $(n_k)_{k\in \N}$ defined by $n_k:=m(|\sqrt{k}|)$  is completely monotone;
 \end{itemize}
 or
 \begin{itemize}
 \item[(H)] $m$ is homogeneous of degree $r<0$, that is $m(k)\eqsim |k|^r$.
 \end{itemize}
\end{description}
 The precise definitions of a completely monotone sequence and the symbol class $S^r(\Z)$ are given in Section~\ref{S:2}. 
 The assumptions on $m$ include dispersive equations with weak dispersion, where the symbol of the Fourier multiplier operator can be either inhomogeneous or homogeneous. They are inspired by the fact that for a large class of evolution equations, solutions which are symmetric at any instant of time are in fact traveling, if the  symbol $m$ is real and even \cite{BEGP,  Ehrn-Holden-Ray}. This fact uncovers the close connection between symmetry and steadiness for waves from the opposite perspective.  
 Examples of well-known water wave model equations falling into the fame of \eqref{eq:Equation} and satisfying assumption (S) or (H), are for instance the Whitham equation, the Burgers--Hilbert equation or the reduced Ostrovsky equation.

 \medskip
 
 Let us turn to our main result. If $u$ is a traveling wave solution of \eqref{eq:Equation}, then $u(t,x)=\phi(x-ct)$, where $c>0$ denotes the wave propagation speed and $\phi$ solves the steady equation
 \begin{equation}\label{eq:steady}
 -c\phi + \frac{1}{2}\phi^2 + L\phi =B
 \end{equation}
 for some  constant of integration $B\in \R$. 
 The present work is not devoted to the existence of solutions of \eqref{eq:steady}, but rather to the symmetry of solutions whenever they exist. Results on the existence of solutions of \eqref{eq:steady} in the periodic setting 
 can be found for instance in \cite{BD, EK, EW} and references therein.
  Our main result reads:
  
   \begin{mainthm}[Symmetry of periodic steady waves]
  Assume that the symbol $m$ of the linear Fourier multiplier operator $L$ in \eqref{eq:Equation} satisfies either assumption \emph{(S)} or \emph{(H)}, and let $\phi$ be a $2\pi$-periodic\footnote{The assumption of $2\pi$-periodic solutions can be replaced by any finite period.}, continuous solution of \eqref{eq:steady}. 
  If $\phi$ satisfies the reflection criterion and one of the following  
  \begin{enumerate}
     \item $\phi <\frac{c}{2}$ ,
    \item 
    $\phi \leq \frac{c}{2}$ and $\phi(x)=\frac{c}{2}$ for a unique $x\in [-\pi,\pi)$
    \end{enumerate}
  holds, then $\phi$ is symmetric and has exactly one crest per period. Moreover,
  \[
  \phi^\prime(x)>0\qquad \mbox{for all}\qquad x\in (-\pi,0),
  \]
  after a suitable translation.
    \end{mainthm}
    
	The proof of the main theorem is given in Theorem \ref{thm:Snew} and Theorem \ref{thm:S}.
   As detailed in the paper, if $\phi<\frac{c}{2}$, then $\phi$ is a smooth solution of \eqref{eq:steady} and we do not need any restriction on the amount and magnitude of its crests (as long as the refection criterion is satisfied). However, if $\phi=\frac{c}{2}$ at some point, then $\phi$ may exhibit a singularity in the form of a cusp or peak, cf. \cite{BD,EW} so that $\phi$ looses its smoothness property. Such a cusp or corner singularity can be compared to a stagnation point on the surface for solutions of the water wave problem, similar as the appearance of a peak in the extremal Stokes wave. As for the Euler equations, this causes difficulties. In our case, we may overcome this problem by the additional assumption that such a singularity occurs at most once per period. Here, we also would like to emphasize that even for solitary solutions as studied in \cite{BEP16}, the proof for the symmetry result fails for the highest wave, unless it is assumed that the highest crest is unique.
   
   \medskip

The proof of the above theorem relies on a weak form of the  method of moving planes, which we apply in a periodic setting to a general nonlocal equation. 
Formally, the action of the nonlocal operator $L$ can be expressed as a convolution with a kernel function $K$, given by a Fourier series with coefficients $(m(k))_{k\in \Z}$.
We show that under condition (S) or (H) the periodic kernel function $K$ can be expressed, as an analog of Bernsteins theorem, in terms of integrals involving Theta or trigonometric functions over measures, respectively.  This result  guarantees that  $K$ is even, integrable and decreasing on a half-period, which forms the foundation for a so-called touching lemma and a boundary point lemma for \eqref{eq:Equation}. Those can be viewed as weak nonlocal counterparts of the maximum principle and Hopf's boundary point lemma for elliptic equations, respectively.

\medskip

We conclude the introduction with the organization of this paper. In Section \ref{S:2} we study the action of the nonlocal Fourier multiplier operator $L$. In particular, we show that if the symbol $m$ satisfies condition (S) or (H), then the action of $L$  corresponds to a convolution operator with an
 even, real and periodic, integrable kernel function $K$. Moreover, $K$ is monotonically decaying on a half-period. Section \ref{S:T} is concerned with the symmetry of traveling wave solutions of \eqref{eq:Equation}. We first study the regularity of periodic steady solutions to \eqref{eq:steady} and then prove a touching lemma and a boundary point lemma. Based on these two lemmas and the structure of smooth or singular wave profiles, we prove the symmetry of regular periodic traveling waves and of the highest periodic wave, respectively. 

\medskip


\section{The action of the nonlocal operator $L$} \label{S:2}
We first introduce some notation. Let $f$ and $g$ be two functions.  We write $f\lesssim g$ ($f\gtrsim g$) if there exists a constant $c>0$ such that $f\leq c g$ ($f\geq cg$). Moreover, we use the notation $f\eqsim g$ whenever $f\lesssim g$ and $f\gtrsim g$. We denote by $\N_0:=\N\cup \{0\}$ the set of natural numbers including zero. 
\subsection{Functional analytic setting}
We denote by $\T:=\R \setminus 2\pi \Z$  the one-dimensional torus, which is identified with $[0,2\pi) \subset \R$. 
 Let  $\mathcal D(\T)=C^\infty(\T)$ and denote by $\mathcal{S}(\Z)$ the space of rapidly decaying functions. Then the (periodic) Fourier transform $\mathcal{F}:\mathcal{D}(\T)\to \mathcal S(\Z)$ is defined by
\[
	(\mathcal{F}f)(k)=\hat f(k):=\frac{1}{{2\pi}}\int_{\T}f(x)e^{-ixk}\,dx.
\]
and any $f\in \mathcal{D}(\T)$ can be written as
\[
f(x)= \sum_{k\in \Z}\hat f(k)e^{ixk}.
\]
By duality, the Fourier transform extends uniquely to $\mathcal{F}:\mathcal{D}^\prime(\T)\to \mathcal S^\prime(\Z)$.
Here, $\mathcal D^\prime(\T)$ and $\mathcal{S}^\prime(\T)$ are the dual spaces of $\mathcal{D}(\T)$ and $\mathcal{S}(\T)$, respectively.
We say a function $f:\T \to \R$ belongs to the space $L^p(\T)$, $1\leq p<\infty$, if and only if
\[
	\|f\|_{L_p}^p:=\int_\T |f|^p(x)\, dx < \infty
\]
and $f\in L^\infty(\T)$ if and only if $\|f\|_\infty:= \esssup_{x\in \T}|f(x)|<\infty$. 
We collect some well-known results on the Fourier transform on $L^p(\T)$ (cf. e.g. \cite[Chapter 3]{G}): If $f\in L^1(\T)$, then the sequence of Fourier coefficients $(\hat f(k))_{k\in \Z}$ is decreasing in $|k|$ with $\lim_{|k|\to \infty}\hat f(k)=0$. If $f,g\in L^1(\T)$ and satisfy $\hat f(k)=\hat g(k)$ for all $k\in \Z$, then $f=g$ almost everywhere. 
If $f,g\in L^2(\T)$, then
\[
\widehat{fg}(k)=\sum_{l\in \Z}\hat f(l)\hat g(k-l)=\sum_{l\in \Z} \hat f(k-l)\hat g(l).
\]


We now introduce the periodic Zygmund spaces on which we perform our  subsequent analysis.
 Let $(\varphi)_{j\geq 0}\subset C_c^\infty(\R)$ be a family of smooth, compactly supported functions satisfying
  \[
 	 \supp \varphi_0 \subset [-2,2],\qquad \supp \varphi_j \subset [-2^{j+1},-2^{j-1}]\cap [2^{j-1},2^{j+1}] \quad\mbox{ for}\quad j\geq 1,
  \]
  \[
 	 \sum_{j\geq 0}\varphi_j(\xi)=1\qquad\mbox{for all}\quad \xi\in\R,
  \]
  and for any $n\in\N$, there exists a constant $c_n>0$ such that 
  \[\sup_{j\geq 0}2^{jn}\|\varphi^{(n)}_j\|_\infty\leq c_n.\]
For $s>0$, the periodic Zygmund space denoted by $ \mathcal{C}^s(\T)$ consists of functions $f$ satisfying
   \[
   	  \|f\|_{\mathcal{C}^s(\T)}:=\sup_{j\geq 0}2^{sj}\left\|\sum_{k\in \Z} e^{ik(\cdot)}  \varphi_j(k)\hat f(k)\right\|_\infty < \infty.
     \]
    
 Eventually, for $\alpha \in (0,1)$, we denote by $C^\alpha(\T)$ the space of $\alpha$-H\"older continuous functions on $\T$. 
  If $k\in \N$ and $\alpha\in (0,1)$, then $C^{k,\alpha}(\T)$ denotes the space of $k$-times continuously differentiable functions whose $k$-th derivative is $\alpha$-H\"older continuous on $\T$. To lighten the notation we write $C^s(\T)=C^{\left \lfloor{s}\right \rfloor, s- \left \lfloor{s}\right \rfloor }(\T)$ for $s\geq 0$. 
As a consequence of Littlewood--Paley theory, we have the relation $\mathcal{C}^s(\T)=C^s(\T)$ for any $s>0$ with $s\notin \N$; that is, the H\"older spaces on the torus are completely characterized by Fourier series. If $s\in \N$, then $C^s(\T)$ is a proper subset of $\mathcal{C}^s(\T)$ and
\[
	C^1(\T)\subsetneq C^{1-}(\T)\subsetneq \mathcal{C}^1(\T).
\]
 Here, $C^{1-}(\T)$ denotes the space of Lipschitz continuous functions on $\T$. For more details we refer to \cite[Chapter 13]{T3}.

\subsection{Fourier multipliers}

A  Fourier multiplier on $\Z$ is a possibly complex valued function that defines a linear operator $L$ via multiplication on the Fourier side, that is
\[
 \widehat{Lf}(k)=m(k)\hat f(k).
\] 
The function $m$ is also called the symbol of the multiplier operator $L$. If $m:\Z \to \C$, let us define the difference operator $\Delta^n$ on $m$ by
\begin{equation}\label{eq:D}
	 \Delta^{n+1} m(k):=\Delta^n m(k+1)-\Delta^nm(k),\qquad n\in \N_0,
\end{equation}
 where $\Delta^0m(k):=m(k)$. Setting $\Delta:=\Delta^1$, we have that $\Delta m(k):=m(k+1)-m(k)$. 
 It is easy to see by induction that
 \[
 \Delta^n m(k)=\sum_{j}^n (-1)^j\binom{n}{j}m(k+n-j).
 \]
 For $r\in \R$ we define the space $S^r(\Z)$ consisting of functions $m:\Z \to \C$ for which
 \[
	 |\Delta^n m(k)|\lesssim_n (1+|k|)^{r-n},\qquad \mbox{for}\quad k\in \Z\quad \mbox{and all}\quad n\in \N_0.
 \]
 If $m\in S^r(\Z)$, we say that $m$ is a symbol of order $r$. The analog definition for functions on the real line states that $m\in S^r(\R)$ if $m\in C^\infty(\R)$ and
 \[
	 |m^{(n)}(\xi)|\lesssim_n(1+|\xi|)^{r-n},\qquad \mbox{for} \quad \xi \in\R \quad \mbox{and all}\quad  n\in \N_0.
 \]
 
 \begin{lem}[\cite{RT}, Lemma 6.2]\label{lem:S}
 If $m\in S^r(\R)$, then the restriction $m|_{\Z}\in S^r(\Z)$.
 \end{lem}

 \medskip
 
 We aim to include in our analysis Fourier multiplier operators with symbols in $S^r(\Z)$, which are bounded, as well as operators with homogeneous symbols. The latter are of the form $m(k)\eqsim|k|^{r}$, where $r<0$. The action of a Fourier multiplier operator with homogeneous symbol on a periodic function is only well-defined for functions with zero mean, that is $\hat f(0)=0$ and 
 \[
	Lf(x)=\sum_{k\neq 0}m(k)\hat f(k)e^{ixk}.
 \]

For this reason, the restriction of a certain function space $X$ to its subset of zero mean functions is going to play an important role and we denote it by $X_0$. We now state a classical Fourier multiplier theorem on Zygmund spaces (e.g. \cite[Proposition 13.8.3]{T3}, \cite[Theorem 2.3 (v)]{AB}):

\begin{satz} \label{prop:FM} Let $r\in \R$. If $m\in S^r(\Z)$,
then the Fourier multiplier $L$ defined by
\[
Lf(x)=\sum_{k\in \Z}m(k)\hat f(k)e^{ixk}
\]
belongs to the space $\mathcal{L}\left(\mathcal{C}^{s}(\T),\mathcal{C}^{s-r}(\T)\right)$ for any $s\geq 0$. Similarly, if $m$ is a homogeneous symbol of order $r$, that is $m(k)\eqsim |k|^r$, then 
then the Fourier multiplier $L$ defined by
\[
Lf(x)=\sum_{k\neq 0}|k|^r\hat f(k)e^{ixk}
\]
belongs to the space $\mathcal{L}\left(\mathcal{C}^{s}_0(\T),\mathcal{C}^{s-r}_0(\T)\right)$ for any $s\geq 0$.
\end{satz}

\medskip

\subsection{Assumptions on the symbol and properties of the convolution kernel} \label{Ss:2}
%
%
%
%
We first recall some results on {completely monotone} sequences (cf. \cite{Guo, Widder}).
 A sequence $(n_k)_{k\in\N_0}$ of real numbers is called \emph{completely monotone} if its elements are nonnegative and
 \[
 	(-1)^n\Delta^nn_k \geq 0\qquad \mbox{for any}\quad n,k\in\N_0,
 \]
 where $\Delta^n$ denotes the difference operator defined in \eqref{eq:D}. In a similar fashion, a smooth function $f:D\subset \R\to \R$ is called \emph{completely monotone} if
 \[
 (-1)^nf^{(n)}(x)\geq 0\qquad \mbox{for any} \quad n\in \N_0, x \in D.
 \]
If $f:\R\setminus\{0\}\to \R$ is even, we say that $f$ is completely monotone if its restriction to the positive half-line $(0,\infty)$ is completely monotone.
 As pointed out in \cite{D}, any nonconstant, completely monotone function on $(0,\infty)$ satisfies the strict inequality $(-1)^nf^{(n)}(x)> 0$. There exists a close relation between completely monotone functions and completely monotone sequences:

 \begin{lem}[\cite{Guo}, Theorem 3 and Theorem 5]\label{lem:CM}
 	Suppose that $f:[0,\infty)\to \R$ is completely monotone, then for any $a\geq 0$ the sequence $(f(an))_{n\in\N_0}$ is completely monotone.
 Conversely, if $(n_k)_{k\in \N}$ is a completely monotone sequence, then there exists a completely monotone interpolation function $f:[1,\infty) \to \R$ such that 
 \[
	 f(k)=n_k\qquad k\in \N.
 \]
 \end{lem}
 
 As an immediate consequence of Lemma \ref{lem:CM}, we observe that any nontrivial monotone sequence $(n_k)_{k\in \N_0}$ is strictly positive for all $k\in \N_0$ and strictly decreasing for all $k\geq 1$.
 
  \medskip

In what follows we impose the following assumption:

 \begin{description}
 \item[Assumption]
%
The symbol $m$ of the Fourier multiplier operator $L$ is real, even, and satisfies either
 \begin{itemize}
 \item[(S)] $m\in S^r(\Z)$ for some $r<0$  and the  sequence $(n_k)_{k\in \N}$ defined by $n_k:=m(|\sqrt{k}|)$  is completely monotone
 \end{itemize}
 or
 \begin{itemize}
 \item[(H)] $m$ is homogeneous of degree $r<0$, that is $m(k)\eqsim |k|^r$.
 \end{itemize}
\end{description}

\medskip

Under assumption (S) or (H) on the symbol, the equation
\begin{equation}\label{eq:Equation_E}
u_t+(u^2+Lu)_x=0
\end{equation}
covers the following widely studied equations: 
 \begin{enumerate}[a)]
 \item The fractional Korteweg--de Vries equation takes the form \eqref{eq:Equation_E} with
 \[
 	m(k)=|k|^r,\qquad r\in \R.
 \]
 The symbol $m$ is homogeneous and satisfies therefore (H), whenever $r<0$.
 The equation corresponds to the Burgers--Hilbert equation for $r=-1$, and to the reduced Ostrovsky equation for $r=-2$.
 \item The Whitham equation takes the form \eqref{eq:Equation_E} with 
 \[
 	m(k)=\sqrt{\frac{\tanh{k}}{k}}.
 \]
 The symbol $m$ is inhomogeneous an can be viewed as the restriction of $M:\R \to \R$, $M(\xi):=\sqrt{\frac{\tanh{\xi}}{\xi}}$ on $\Z$. The function $M$ belongs to $S^{-\frac{1}{2}}(\R)$ and $\xi \to M(|\sqrt{\xi}|)$ is completely monotone on $(0,\infty)$ as proved in \cite{EW}. Now, Lemma \ref{lem:S} and Lemma \ref{lem:CM} imply that $m$ satisfies assumption (S) for $r=-\frac{1}{2}$.
 \item The inhomogeneous counter part of the fractional Korteweg--de Vries family, takes the form \eqref{eq:Equation_E} with
 \[
 	m(\xi)=(1+k^2)^{\frac{r}{2}},\qquad r\in \R.
 \] 
 The symbol $m$ is inhomogeneous and satisfies (S) for $r<0$. Again, this is a direct consequence of Lemma~\ref{lem:S} and Lemma \ref{lem:CM} and the fact that $n:(0,\infty)\to \R$ defined by $n(\xi)=(1+\xi)^{\frac{r}{2}}$ is completely monotone on $(0,\infty)$ for $r<0$.
 \end{enumerate}
 
 The action of the nonlocal operator $L$ can be expressed as a convolution with a kernel function $K$, which takes the form
 \[
 K(x)=\sum m(k)\cos(xk),
 \]
 and $L\phi = K*\phi$. The sum above is taken over $\Z$ or $\Z\setminus\{0\}$ depending on whether $m$ is an inhomogeneous or a homogeneous symbol, respectively.

%
 
 \medskip
%
%

The following  discrete analog of Bernstein's theorem on completely monotone functions will be used to prove the monotonicity of $K$ on the half-period $(0,\pi)$.
\begin{thm}[\cite{Widder}, Theorem 4a]\label{thm:B}
	A sequence $(n_k)_{k\in\N_0}$ of real numbers is  completely monotone if and only if
	\[
		n_k=\int_0^1 t^k d\sigma(t),
	\]
	where $\sigma$ is nondecreasing and bounded for $t\in[0,1]$.
\end{thm}

We now prove the the integrability of $K$ and its monotonicity on $(0,\pi)$, which is crucial for the touching lemma and the boundary point lemma in the next section.
\begin{thm}[Properties of $K$] \label{thm:P} The periodic kernel $K$ satisfying the assumption (S) or (H) is even, real-valued and smooth on $\T\setminus\{0\}$. Moreover, $K\in L^1(\T)$ is decreasing on $(0,\pi)$.
 If the symbol $m$ is homogeneous of degree $r<0$, then
\[
K(x)=\int_0^1 \left(\frac{2(\cos(x)-t)}{1-t\cos(x)+t^2}+a_0(t)\right)\,d\sigma(t),
\]
where $\sigma:[0,1]\to \R$ is a nondecreasing and bounded function depending on $m$, and $a_0\in L^\infty(0,1)$. If the symbol $m$ is inhomogeneous and satisfies (S), then
\[
K(x)=\int_0^1\Theta_3\left(\frac{x}{2},u\right)\,d\nu(u),
\]
where $\nu:[0,1]\to \R$ is a nondedcreasing and bounded function depending on $m$, and $\Theta_3$ is the third Theta function.
\end{thm}

\begin{proof} In fact, if $m$ is a homogeneous symbol of degree $r<-1$ the claim is proved in \cite[Theorem 3.6]{BD} and it is straightforward to adapt the proof for the range $r\in [-1,0)$. We therefore skip the details. 
 Assume that $m$ is an inhomogeneous symbol satisfying (S). Since $m$ is even and real-valued, also $K$ is even and real-valued. The integrability of $K$ follows from the decay property of $m$ and \cite[Theorem 5.13]{Boa}. Now, we prove that $K$ is smooth on $\T\setminus\{0\}$ and decreasing on the half-period $(0,\pi)$.
  If $n:=m(|\sqrt{\cdot}|)$, then Lemma \ref{lem:CM} implies that  $n_k:=n(k)$ build a completely monotone sequence $(n_k)_{k\in \N_0}$. In view of Theorem \ref{thm:B}, there exists a nondecreasing and bounded function $\nu$ such that
\[
n_k=\int_0^1 u^kd\nu(u),\qquad k\geq 0.
\]
We have that $m(k)=n(k^2)$ and thus
\[
m(k)=\int_0^1 u^{k^2}d\nu(u),\qquad \mbox{for all}\qquad k\geq 0.
\]
Consider $u^{k^2}$ as  Fourier coefficients for some function $f(u,x)$, that is
\[
u^{k^2}=\int_{\T}f(u,x)e^{-ikx}\qquad \mbox{for}\qquad f(u,x)=\sum_{k\in \Z} u^{k^2}e^{ixk}.
\]
The latter sum is also known as the third Theta function:
\[
	f(u,x)=\sum_{k\in \Z} u^{k^2}e^{ixk} =: \Theta_3\left(\frac{x}{2},u\right).
\]
We conclude that
\[
	m(k)= \int_0^1 \int_{\T}\Theta_3 \left(\frac{x}{2},u\right)e^{-ikx} \,dx\,d\nu(u)= \int_{\T}\int_0^1 \Theta_3 \left(\frac{x}{2},u\right)\,d\nu(u)e^{-ixk}\,dx.
\]
Since $(m(k))_{k\in \Z}$ form the Fourier coefficients of $K$, we deduce that
\[
	K(x)=\int_0^1 \Theta\left( \frac{x}{2},u \right)\,d\nu (u).
\]
From here, we obtain all claimed properties of $K$ relying on the properties of the Theta function. In particular, we have that for all $u\in (0,1)$ the function $\Theta(\frac{\cdot}{2},u)$ is even, positive on $\T$ and $\frac{d}{dx}\Theta\left(\frac{x}{2},u\right)< 0$ for all $x\in (0,\pi)$ and all $u\in (0,1)$. Hence $K$ is smooth away from the origin and $K^\prime(x)< 0$ for all $x\in (0,\pi)$.
\end{proof}

\begin{lem}\label{behaviour of K at origin}
The map
$x\mapsto \sin(x)K(x)$ belongs to $L^\infty(\T)$ and
$
	\lim_{|x|\to 0} xK(x)=0.
$
\end{lem}

\begin{proof}
Using the identity $\sin(x)\cos(xk)= \frac{1}{2}\left(\sin(x(k+1))-\sin(x(k-1))\right))$, we get
\[
	\sin(x)K(x)=c_h\sin(x)+\sum_{k=1}^\infty m(k)\left(\sin(x(k+1))-\sin(x(k-1))\right),
\]
where $c_h=0$ if $m$ is a homogeneous symbol and $c_h= m(0)$ if $m$ satisfies (S). The sum above can be written as
\begin{align*}
\sum_{k=1}^\infty m(k)\left(\sin(x(k+1))-\sin(x(k-1))\right)&=  \sum_{k=2}^\infty m(k-1)\sin(xk)-\sum_{k=0}^\infty m(k+1)\sin(xk)\\
&= m(2)\sin(x)+\sum_{k=2}^\infty \left( m(k-1)-m(k+1)\right)\sin(xk).
\end{align*}
We have that $\sum_{k=2}^\infty a_k\sin(xk)$ converges uniformly on $\T$ iff  $\lim_{k\to \infty} ka_k=0$ (cf. \cite{Boa}). In view of either assumption (S) or (H), we have that $m(k-1)-m(k+1)\leq \Delta m(k)\lesssim |k|^{r-1}$ for $k\geq 2$. Since $r<0$, it is clear that $\lim_{k\to \infty}k(m(k-1)-m(k+1))=0$ and the series $\sum_{k=2}^\infty(m(k-1)-m(k+1))\sin(xk)$ converges uniformly on $\T$. We deduce that $|\sin(x)K(x)|\lesssim 1 + \sum_{k=2}|k|^{r-1}\lesssim 1$ for all $x\in \T$, which implies that $x\mapsto \sin(x)K(x)$ belongs to $L^\infty(\T)$. Moreover,
\begin{align*}
	\lim_{|x|\to 0}xK(x) &= \lim_{|x|\to 0}\sin(x)K(x) = \lim_{|x|\to 0}\lim_{n\to \infty}\sum_{k=2}^n \left( m(k-1)-m(k+1)\right)\sin(xk)\\
	&= \lim_{n\to \infty}\lim_{|x|\to 0}\sum_{k=2}^n \left( m(k-1)-m(k+1)\right)\sin(xk)\\
	& =0.
\end{align*}
Here, we used that $\lim_{|x|\to 0}\frac{\sin(x)}{x}=1$.
\end{proof}

\begin{remark}\label{rem:R}
\normalfont
\begin{enumerate}[a)]
\item 
 Requiring complete monotonicity of $(m(|\sqrt{k}|))_{k\in \N_0}$ instead of assuming that $(m(k))_{k\in \N_0}$ is completely monotone actually broadens the class of admissible symbols, since the composition $n:=m \circ \sqrt{\cdot}$ is completely monotone whenever $m$ is completely monotone.
\item 
The proof of Theorem \ref{thm:P} reveals that whenever the sequence $(n_k)_{k\in \N_0}$ defined by $n_k:=m(|\sqrt{k}|)$ is completely monotone, the corresponding convolution kernel given by $K(x)=\sum m(k)\cos(xk)$ is decreasing on the half-period $(0,\pi)$. 
It turns out to be a  nontrivial task to find a one-to-one correspondence between properties of the Fourier coefficients and decay on a half-period of the corresponding Fourier series. The assumptions we impose on the symbol $m$  are chosen to be fairly mild to include a wide range of admissible symbols while still enabling a straightforward verification in specific examples. 
\item
Assuming that the function $n:=m(|\sqrt{\cdot}|):(0,\infty)\to (0,\infty)$ is not only bounded and completely monotone, but  also extends additionally to an analytic function on $\C\setminus (-\infty,0]$ with  $\Ima z\cdot\Ima n(z)\leq 0$, then the convolution kernel $K$ inherits the property of being completely monotone (cf. \cite[Theorem. 2.9, Proposition 2.20, Remark 3.4]{EW}). 
\end{enumerate}
\end{remark}

\bigskip

\section{Symmetry of traveling waves}
\label{S:T}

%
In this section we prove the main theorem on the symmetry of periodic traveling waves for nonlinear dispersive equations of the from \eqref{eq:Equation} where the symbol $m$ satisfies either assumption (S) or (H), and the wave profile satisfies the reflection criterion, which we recall for convenience.
\begin{description} 
\item[Reflection criterion]
A $2\pi$-periodic, continuous function  is said to satisfy the \emph{reflection criterion} if  there exists $\lambda_*\in \T$ such that
\[
	\phi(x)>\phi(2\lambda_*-x)\qquad \mbox{for all}\qquad x\in (\lambda_*, \lambda_*+ \pi).
\]
\end{description}
Taking the  ansatz $u(x,t)= \phi(x-ct)$, where $c>0$ denotes the speed of the right--propagating wave,  equation \eqref{eq:Equation} transforms after integration  to 
\begin{equation}\label{WT}
-c\phi + L\phi+\phi^2=B,
\end{equation}
where $B\in \R$ is a constant of integration. If $m$ is an inhomogenous symbol, then there exists a Galilean shift of variables 
 \[
	 \phi \mapsto \phi+ \gamma, \qquad c\mapsto c+2\gamma, \qquad B\mapsto B+\gamma(m(0)-c-\gamma),
 \]
 which allows us to set the integration constant $B$ to zero. This choice corresponds to a solution with possible different speed and elevation, but the form of solutions remains intact. If on the other hand $m$ is a homogeneous symbol, we assume that $\phi$ is a function of zero mean, which determines the integration constant to be $B= \frac{1}{2\pi}\widehat{\phi^2}(0)$. In what follows, we consider the equation
 \begin{equation}\label{eq:new}
 -c\phi +L\phi +\phi^2 =B_h,
 \end{equation}
 where $B_h=0$ if $m$ is inhomogeneous and $B_h=\frac{1}{2\pi}\widehat{\phi^2}(0)$ if $m$ is homogeneous.
 
\subsection{Regularity of traveling waves}\label{Ss:reg}
 If the symbol of the operator $L$ is homogeneous, we work on $X_0$,  the restriction of a function space $X$ to its subset of zero mean functions. Let $\phi \in L^\infty(\T)$ or $\phi \in L^\infty_0(\T)$ be a  $2\pi$-periodic solution of \eqref{eq:new}.
For the clarity of presentation, we use in the sequel the following convention: Whenever it is clear from the context the index zero is suppressed, that is we simply write $X$ and mean $X_0$ if $m$ is a homogeneous symbol.

\begin{satz}\label{prop:reg}
Let $\phi\leq \frac{c}{2}$ be a bounded solution of \eqref{eq:new}. Then $\phi$ is smooth on any open set where $\phi<\frac{c}{2}$. 
\end{satz}
\begin{proof}
 Assume first that $\phi<\frac{c}{2}$ uniformly on $\T$ and  $\phi \in \mathcal{C}^s(\T)$ for some $s\geq 0$. Equation \eqref{eq:new} can be written as
 \begin{equation}\label{eq:formulation}
 B_h+\frac{c^2}{4}-\left(\frac{c}{2}-\phi \right)^2 =L\phi.
 \end{equation} 
 Due to our assumptions on the symbol $m$ and Proposition \ref{prop:FM}, the Fourier multiplier operator $L$ is a smoothing operator of order $-r$  and $L\phi \in \mathcal{C}^{s-r}(\T)$.
Moreover, for $s-r>0$ the Nemytskii operator
\begin{align*}
f\mapsto \frac{c}{2} - \sqrt{B_h+\frac{c^2}{4}- f}
\end{align*}
maps $  \mathcal{C}^{s-r}(\T)$ into itself if $f<B_h+\frac{1}{4}c^2$. From \eqref{eq:formulation} we see immediately that $L\phi<B_h+\frac{c^2}{4}$ if $\phi<\frac{c}{2}$.
Thus, we may take the square root to obtain that
\[
\phi = \frac{c}{2}-\sqrt{B_h+\frac{c^2}{4}-L\phi} \in \mathcal{C}^{s-r}(\T).
\]
Hence, an iteration argument guarantees that $\phi\in C^{\infty}(\T)$. Since any Fourier multiplier commutes with the translation operator, we actually have that $\phi \in C^\infty(\R)$.

Now, let $U\subset \R$ be an open subset of $\R$ on which $\phi <\frac{c}{2}$. Then, we can find an open cover $U=\cup_{i\in I}U_i$, where for any $i\in I$ we have that $U_i$ is connected and satisfies $|U_i|<2\pi$. Due to the translation invariance of \eqref{eq:new} and the previous part, we obtain that $\phi$ is smooth on $U_i$ for any $i\in I$. Since $U$ is the union of open sets, the assertion follows.
\end{proof}

%
%

The following lemma confirms that a non-smooth structure may appear at a crest of height $\frac{c}{2}$.
\begin{satz}\label{prop:C1} If $\phi$ is a nontrivial even, bounded solution of \eqref{eq:new} with a single crest per period and $\max \phi = \frac{c}{2}$, then $\phi$ does not belong to the class $C^1(\T)$.
\end{satz}
\begin{proof} 
 Assume  on the contrary that $\phi\in C^1(\T)$ is an even solution of \eqref{eq:new} with $\max \phi = \phi(0)=\frac{c}{2}$. Then  $L\phi$ belongs to $\mathcal{C}^{1-r}(\T)$ with $(L\phi)^\prime(0)=0$ and 
 \[
\left(\frac{\frac{c}{2}-\phi}{x}\right)^2=\frac{L\phi(0)-L\phi(x)}{x^2}=\frac{(L\phi)^\prime(\xi)-(L\phi)^\prime(0)}{x}\qquad \mbox{for some}\qquad \xi \in [0,x].
 \]
 Since $\phi\in C^1$, the left-hand side above tends to zero as $x\to 0$. We deduce that $(L\phi)^{\prime\prime}(0)=0$. The action of $L$ is given by convolution with $K$, that is
 \[
 (L\phi)^{\prime\prime}(0)=-2\int_{-\pi }^0 K ^\prime(y)\phi^\prime(y)\,dy=-c_0<0,
\]
for some $c_0>0$, which is a contradiction. Here, we used the symmetry of $K$ and $\phi$ and that $K ^\prime\phi^\prime\gneq 0$ on $[-\pi ,0]$, since $K^\prime > 0$ and $\phi^\prime\gneq 0$ on $(-\pi,0)$, due to Proposition \ref{thm:P} and our assumption on $\phi$.
\end{proof}
\begin{remark}
\normalfont
If $\phi\leq \frac{c}{2}$ is a solution of \eqref{eq:new} with $\max \phi=\frac{c}{2}$, we refer to $\phi$ as a so-called \emph{highest wave}. 
\end{remark}
%
%
%
%
\subsection{Symmetry of traveling waves}
 We start with a touching lemma, which is similar to the one formulated in \cite[Lemma 4.3]{EW} for solitary waves. A solution $\phi$ of \eqref{eq:new} is called a \emph{supersolution} if
\[c\phi \geq L\phi+\phi^2-B_h\]
and a \emph{subsolution} if the inequality sign above is replaced by $\leq$.

\begin{lem}[Touching lemma within one period]\label{touching lemma}
Let $\phi$ be a bounded $2\pi $-periodic super- and $\bar{\phi}$ a bounded $2\pi $-periodic subsolution of \eqref{eq:new}. If $\phi \geq \bar{\phi}$ on $[\lambda,\lambda +\pi ]$ and $\phi-\bar{\phi}$ is odd with respect to $\lambda$, then either
\begin{itemize}
\item[$\bullet$] $\phi=\bar{\phi}$ on $\R$ or
\item[$\bullet$] $\phi > \bar{\phi}$  with $\phi+\bar \phi<c$ on $(\lambda, \lambda+\pi )$.
\end{itemize}
\end{lem}
%
%
\begin{proof}
Let $\phi$ and $\bar \phi$ be a super- and subsolution of \eqref{WT}, respectively, with $\phi \geq \bar \phi$ for all $x\in [\lambda, \lambda+\pi ]$ and $\phi-\bar \phi$ is odd with respect to $\lambda$. Set $w:=\phi-\bar \phi$. Then $w$ is a $2\pi $-periodic function, which is odd with respect to $\lambda$ and $w(x)\geq 0$ for  $x\in (\lambda, \lambda+\pi )$. The function $w$ solves the equation
\[
	cw(x)\geq Lw(x)+w(x)(\phi+\bar \phi)(x),
\]
which is equivalent to
\begin{equation}\label{eq:t1}
	(c-(\phi+\bar \phi)(x))w(x)\geq Lw(x).
\end{equation}
Assume that $w$ is not identical zero, but there exist a point $\bar x \in (\lambda, \lambda+\pi )$, such that either $w(\bar x)=0$ or $(\phi+\bar \phi)(\bar x)\geq c$. Then, we obtain from \eqref{eq:t1} that 
\begin{equation}\label{eq:con}
Lw(\bar x)\leq 0.
\end{equation}
Note that
\begin{align*}
Lw(\bar x)=\int_{\lambda-\pi }^{\lambda+\pi }K(\bar x-y)w(y)\,dy=\int_{\lambda}^{\lambda+\pi }\left[K(\bar x-y)-K(\bar x+y-2\lambda) \right]w(y)\,dy.
\end{align*}
  Now we split the integral above in the following way:
\begin{align}\label{eq:I}
\begin{split}
Lw(\bar x)=&\int_{\lambda}^{\bar x}\left[K(\bar x-y)-K(\bar x+y-2\lambda) \right]w(y)\,dy\\
&+\int_{\bar x}^{\lambda +\pi }\left[K(\bar x-y)-K(2\lambda-\bar x-y) \right]w(y)\,dy,
\end{split}
\end{align}
keeping in mind that $\lambda<\bar x <\lambda+\pi $.
The function 
\[
G_p(y):=K(\bar x-y)-K(\bar x+y-2\lambda) 
\]
is $2\pi $-periodic odd with respect to $\lambda$. Notice that $w$ is nonnegative on the half period $(\lambda, \lambda+\pi )$ and
\[
\lim_{|y-\bar x|\to 0} G_p(y)>0 \qquad \mbox{and}\qquad G_p(\lambda)=G_p(\lambda+\pi )=0.
\]

\begin{center}
		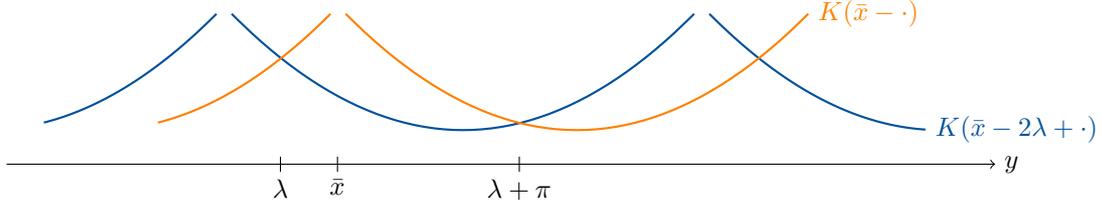
\begin{figure}[h]
		\centering
		\begin{tikzpicture}[scale=1]
		\draw[->] (-6,-1) -- (7,-1) node[right] {$y$};
		\draw[-] (-2.4, -1.1)--(-2.4,-0.9) node[below=5pt]{$\lambda$};
		\draw[-] (-2.4+pi, -1.1)--(-2.4+pi,-0.9) node[below=5pt]{$\lambda+\pi $};
		\draw[-] (-1.65, -1.1)--(-1.65,-0.9) node[below=5pt]{$\bar x$};
		\draw[domain=-pi+0.1:pi-0.1,smooth,variable=\x,luh-dark-blue, thick] plot ({\x},{(3*\x*\x-pi^2)/18});
		\draw[domain=-1.7*(pi+0.1):-(pi+0.1),smooth,variable=\x,luh-dark-blue, thick] plot ({\x},{((3*(\x+2*pi)*(\x+2*pi)-pi^2)/18});
		\draw[domain=pi+0.1:2*(pi-0.1),smooth,variable=\x,luh-dark-blue, thick] plot ({\x},{((3*(\x-2*pi)*(\x-2*pi)-pi^2)/18}) node[right]{$K(\bar x-2\lambda +\cdot)$};
		\draw[xshift = 1.5cm, domain=-pi+0.1:pi-0.1,smooth,variable=\x,orange, thick] plot ({\x},{(3*\x*\x-pi^2)/18}) node[right]{$K(\bar x -\cdot)$};
				\draw[xshift = 1.5cm,domain=-1.7*(pi+0.1):-(pi+0.1),smooth,variable=\x,orange, thick] plot ({\x},{((3*(\x+2*pi)*(\x+2*pi)-pi^2)/18}) ;
		\end{tikzpicture}
		\caption*{Illustration: $G_p(y)=K(\bar x-y)-K(\bar x -2\lambda+y)>0$ on $(\lambda,\lambda+\pi )$.}
		\end{figure}
	\end{center}

We aim to show that $G_p(y)>0$ on $(\lambda,\lambda+\pi )$. Set $z=y-\bar x $ and $v=2(\bar x-\lambda)$, then $G_p(y)=0$ if and only if $K(z)=K(z+v)$. In view of the symmetry  of $K$ and its monotonicity on $(0,\pi )$, it is clear that $K(z)=K(z+v)$ if and only if $v\in 2\pi \Z$ or $v \in -2z+2\pi \Z$. We have $v=2(\bar x-\lambda)\in (0,2\pi )$, therefore $v\notin 2\pi \Z$. Moreover, $v \in -2z+2\pi \Z$ if and only if there exists $n\in \Z$ such that
\[
2(\bar x-\lambda)=-2(y-\bar x)+2\pi n,	
\]
which is equivalent to $y=\lambda+2\pi n$. We deduce that $G_p(y)>0$ on $(\lambda,\lambda+\pi )$.
%
%
Recalling \eqref{eq:I} and $w(y)\geq 0$ on $(\lambda, \lambda+\pi )$, we obtain that $K*w(\bar x)>0$, which is a contradiction to \eqref{eq:con}.

\end{proof}

While the touching lemma is related to a strong maximum principle, the following lemma plays a role as the Hopf boundary point lemma does for elliptic equations.

\begin{lem}[Boundary point lemma]\label{boundary lemma} Let $\phi, \bar \phi\in C^1(\T)$ be two $2\pi $-periodic solutions of \eqref{eq:new}. If $\phi \geq \bar \phi$ on $[\lambda, \lambda + \pi ]$ and $\phi-\bar \phi$ is odd with respect to $\lambda$, then either
\begin{itemize}
\item $\phi = \bar \phi$ on $\R$, or
\item $ (\phi-\bar \phi) ^\prime(\lambda)>0.$
\end{itemize}
\end{lem}

\begin{proof} If $\phi$ and $\bar \phi$ are two solutions of \eqref{eq:new}, then
\[
c(\phi-\bar \phi)(x) = K*(\phi-\bar \phi)(x)+\phi^2(x)-\bar \phi^2(x).
\]
Taking the derivative at $x=\lambda$ yields
\begin{equation}\label{eq:bnd}
[c-(\phi+\bar \phi)(\lambda)](\phi-\bar \phi)^\prime(\lambda)=K*(\phi-\bar{\phi})^\prime(\lambda).
\end{equation}
Set $w=\phi-\bar \phi$ and consider the convolution $K*w^\prime (\lambda)$. Since $w^\prime$ is symmetric with respect to $\lambda$ and $K$ is even, we deduce that
\begin{align*}
K*w^\prime(\lambda)&=\int_{\lambda-\pi }^{\lambda+\pi }K(\lambda-y)w^\prime(y)\, dy=2\int_{\lambda}^{\lambda+\pi }K(\lambda-y)w^\prime(y)\, dy.
\end{align*}
Using that $K$ is smooth on $(-\pi ,\pi )$ away from the origin, we integrate by parts on the interval $[\lambda+\e,\lambda+\pi ]$ and obtain that
\begin{align*}
K*w^\prime(\lambda)&= 2\left( \int_{\lambda}^{\lambda+\e}K(\lambda-y)w^\prime(y)\, dy +[K(\lambda-y)w(y)]_{y=\lambda+\e}^{\lambda+\pi }+ \int_{\lambda+\e}^{\lambda+\pi } K^\prime(\lambda-y)w(y)\, dy\right)
\end{align*}
Because $w^\prime$ is continuous and $K$ is integrable, the first integral on the right-hand side vanishes as $\e \to 0$.
Due to the regularity and symmetry of $w$, we have $w(\lambda +\e)=O(\e)$.  Moreover $K(\e)=o(\e^{-1})$ by Lemma \ref{behaviour of K at origin} so that the boundary term 
\[
[K(\lambda-y)w(y)]_{y=\lambda+\e}^{\lambda+\pi }=K(\pi )w(\lambda+\pi )-K(\e)w(\lambda+\e)\to 0
\] 
as $\e\to 0$, where we used that $w(\lambda+\pi )=w(\lambda)=0$. Hence
\begin{align*}
K*w^\prime(\lambda)=2\lim_{\e \to 0} \int_{\lambda+\e}^{\lambda+\pi } K^\prime(\lambda-y)w(y)\, dy.
\end{align*}
In view of $w \geq 0$ on $[\lambda,\lambda+\pi ]$ and $K$ being increasing on the half-period $(-\pi ,0) $, we arrive at
\[
K*w^\prime(\lambda)=2\lim_{\e \to 0} \int_{\lambda+\e}^{\lambda+\pi } K^\prime(\lambda-y)w(y)\, dy>0,
\]
unless $\phi = \bar \phi$. Now \eqref{eq:bnd} implies that
\[
(\phi-\bar \phi)^\prime(\lambda)>0.
\]
\end{proof}

\begin{remark}
\normalfont
The proof of the following theorems rely on a weak form of the method of moving planes, which we apply in a nonlocal setting. 
Due to the periodicity of the solution and therefore the lack of decay at infinity, we impose the aforementioned reflection criterion
in order to guarantee that the method of moving planes can be started at some point $x\in \T$. 
Notice that whenever the wave profile is a priori assumed to be \emph{monotone} in the sense that it has only one crest per period, our assumption is satisfied. 

\begin{center}
\begin{figure} [h!]
\begin{tikzpicture}[scale=0.7]
\draw (-2,0)--(8,0);
\draw[-,black](0,0.1)--(0,-0.1) node[below] {$-\pi$};
\draw[-,black, gray](2.5,0.1)--(2.5,-0.1) node[below] {$\lambda_*$};
\draw[-,dashed, gray] (2.5,0)--(2.5,1.5);
\draw[-,black](6,0.1)--(6,-0.1) node[below] {$\pi$};
\draw[-,black](3,0.1)--(3,-0.1) node[below] {$0$};
\draw[luh-dark-blue, very thick] plot [smooth] coordinates {(-1,1) (0,0.5)  (4,2)  (6,0.5) (7,0.8)};
\begin{scope}[xscale=-1]
  \draw[gray, dashed, xshift=-5cm] plot [smooth] coordinates {(-1,1) (0,0.5)  (4,2)  (6,0.5) (7,0.8)}; 
\end{scope}
\end{tikzpicture}
\qquad
\begin{tikzpicture}[scale=0.65]
\draw (-2,0)--(8,0);
\draw[-,black](0,0.1)--(0,-0.1) node[below] {$-\pi$};
\draw[-,black](6,0.1)--(6,-0.1) node[below] {$\pi$};
\draw[-,black](3,0.1)--(3,-0.1) node[below] {$0$};
\draw[luh-dark-blue, very thick] plot [smooth] coordinates {(-1,1.3) (0,0.5) (0.8,1) (1.3,0.8) (2.1,1.4) (2.5, 1.8) (3.2, 1.6) (5.1,1.4)  (6,0.5) (6.8,0.8)};
\begin{scope}[xscale=-1]
  \draw[gray, dashed, xshift=-4.5cm] plot [smooth] coordinates {(-1,1.3) (0,0.5) (0.8,1) (1.3,0.8) (2.1,1.4) (2.5, 1.8) (3.2, 1.6) (5.1,1.4)  (6,0.5) }; 
\end{scope}
\draw[-,black, gray](2.27,0.1)--(2.27,-0.1) node[below] {$\lambda_*$};
\draw[-,dashed, gray] (2.27,0)--(2.27,1.5);
\end{tikzpicture}
\caption{Any monotone profile satisfies reflection criterion (left). 
The illustration of a nonmonotone wave profile satisfying  reflection criterion (right). The gray dashed curves are the reflections about the axis $x=\lambda_*$.}\label{F:profile}
\end{figure}
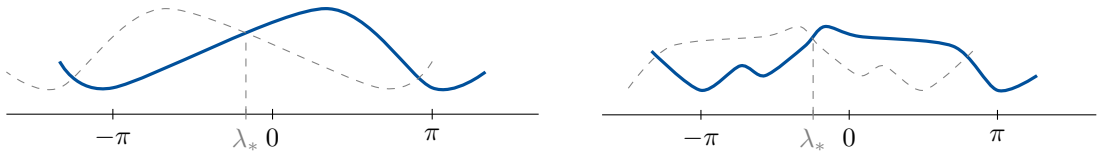
\end{center}
\end{remark}

\begin{thm}[Symmetry of traveling waves]\label{thm:Snew}
Let $\phi<\frac{c}{2}$ be a $2\pi$-periodic, bounded solution of \eqref{eq:new}, and satisfies the reflection criterion. 
Then $\phi$ is symmetric and has exactly one crest per period. Moreover,
\[
\phi^\prime(x)>0\qquad \mbox{for all}\qquad x\in (-\pi,0),
\]
after a suitable translation.
\end{thm}

\begin{proof}
Since $\phi<\frac{c}{2}$, Proposition \ref{prop:reg} implies that $\phi$ is a smooth solution.
In order to prove the theorem, we suppose that $\phi$ is not symmetric. After possible translation, we may assume that a global minimum (trough) of the periodic wave is located at $x=-\pi$. Then $\lambda_*\in (-\pi,0]$. Indeed, if $\lambda_*\in (0,\pi]$, then the assumption
\[
	\phi(x)>\phi(2\lambda_*-x)\qquad \mbox{for all}\quad x \in (\lambda_*, \lambda_*+\pi)
\] 
would yield the contradiction $\phi(-\pi)=\phi(\pi)>\phi(2\lambda_*-\pi)$, since the global minimum of $\phi$ is attended at $x=-\pi$. Let $w_\lambda:\R\rightarrow \R$ be the \emph{reflection function} around $\lambda$ given by
\[w_\lambda(x):= \phi(x)-\phi(2\lambda-x).\] 
 Due to the symmetry of $K$ the function $\phi(2\lambda-\cdot)$ is a solution of \eqref{eq:new} whenever $\phi$ is a solution. 
Set 
	\begin{equation}\label{eq:lambda}\lambda_0:= \sup \{\lambda \in [\lambda_*,0 ] \mid w_\lambda(x) > 0 \; \mbox{ for all } \; x\in (\lambda,\lambda+\pi )\}.
	\end{equation}
	Notice that such $\lambda_0\geq \lambda_*$ exists, because by assumption
	\[
		w_{\lambda_*}(x)=\phi(x)-\phi(2\lambda_*-x)>0 \qquad \mbox{for all}\qquad x\in(\lambda_*,\lambda_*+\pi).
	\]
	We have that $w_{\lambda_0}$ satisfies
		\begin{itemize}
		\item[i)] $w_{\lambda_0}({\lambda_0})=0$; 
		\item[ii)] $w_{\lambda_0}$ is odd with respect to $\lambda_0$, that is $w_{\lambda_0}(\cdot)=-w_{\lambda_0}(2\lambda_0-\cdot)$;
		\item[iii)] $w_{\lambda_0}\geq 0$ in $[\lambda_0, \lambda_0+\pi ]$ and $w_{\lambda_0}\leq 0$ in $[\lambda_0+\pi , \lambda_0+2\pi ]$.
		\end{itemize}
Let us consider $w_\lambda$ for $\lambda\geq \lambda_*$. Starting at $\lambda=\lambda_*$, we move the plane $\lambda$ about which the wave profile is reflected forward as long as $w_{\lambda}>0$ on $(\lambda,\lambda+\pi )$. Clearly, this process stops at or before the first crest in $[\lambda_*,0]$ at $\lambda=\lambda_0$. In fact one of three occasions will occur (cf. Figure \ref{F:Alternatives}): Either there exists $\bar x \in (\lambda_0,\lambda_0+\pi )$ such that $w_{\lambda_0}(\bar x)=0$ as i.e., in (a); or we reach a crest at $x=\lambda_0$ as i.e., in (b); or we reach a  trough at $\lambda_0+\pi $ as i.e.,  in (c). 
\begin{center}
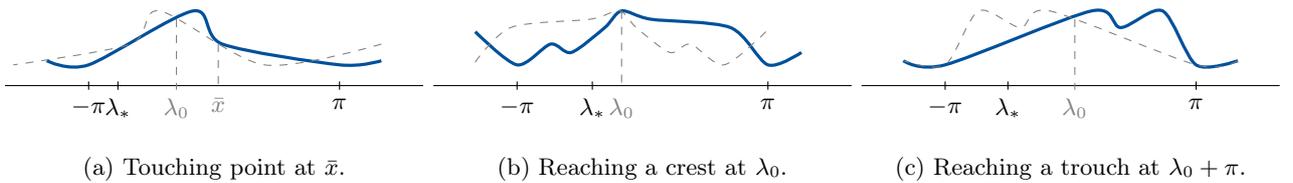
\begin{figure}[h!]
\begin{tikzpicture}[scale=0.55]
\small
\draw (-2,0)--(8,0);
\draw[-,black](0,0.1)--(0,-0.1) node[below] {$-\pi$};
\draw[-,black](6,0.1)--(6,-0.1) node[below] {$\pi$};
\draw[luh-dark-blue, very thick] plot [smooth] coordinates { (-1,0.6)(0,0.5) (2.5,1.8) (3.2,1) (6,0.5) (7,0.6) };
\begin{scope}[xscale=-1]
  \draw[gray, dashed, xshift=-4.2cm] plot [smooth] coordinates {(-2.8,1) (0,0.5) (2.5,1.8) (3.2,1) (6,0.5) }; 
\end{scope}
\draw[-,black, black](0.7,0.1)--(0.7,-0.1) node[below] {$\lambda_*$};
\draw[-,dashed, gray] (2.1,1.6)--(2.1,-0.1)node[below] {$\lambda_0$};
\draw[-,dashed, gray] (3.1,1)--(3.1,-0.1)node[below] {$\bar x$};
\node at (3,-2) {(a) Touching point at $\bar x$.};
\end{tikzpicture}
\begin{tikzpicture}[scale=0.55]
\small
\draw (-2,0)--(8,0);
\draw[-,black](0,0.1)--(0,-0.1) node[below] {$-\pi$};
\draw[-,black](6,0.1)--(6,-0.1) node[below] {$\pi$};
\draw[luh-dark-blue, very thick] plot [smooth] coordinates {(-1,1.3) (0,0.5) (0.8,1) (1.3,0.8) (2.1,1.4) (2.5, 1.8) (3.2, 1.6) (5.1,1.4)  (6,0.5) (6.8,0.8)};
\begin{scope}[xscale=-1]
  \draw[gray, dashed, xshift=-4.95cm] plot [smooth] coordinates {(-1,1.3) (0,0.5) (0.8,1) (1.3,0.8) (2.1,1.4) (2.5, 1.8) (3.2, 1.6) (5.1,1.4)  (6,0.5) }; 
\end{scope}
\draw[-,black, black](1.8,0.1)--(1.8,-0.1) node[below] {$\lambda_*$};
\draw[-,dashed, gray] (2.5,1.8)--(2.5,-0.1)node[below] {$\lambda_0$};
\node at (3,-2) {(b) Reaching a crest at $\lambda_0$.};
\end{tikzpicture}
\begin{tikzpicture}[scale=0.55]
\small
\draw (-2,0)--(8,0);
\draw[-,black](0,0.1)--(0,-0.1) node[below] {$-\pi$};
\draw[-,black](6,0.1)--(6,-0.1) node[below] {$\pi$};
\draw[luh-dark-blue, very thick] plot [smooth] coordinates { (-1,0.6)(0,0.5) (3.5,1.8) (4.2,1.4)  (5.2, 1.8)(6,0.5) (7,0.6) };
\begin{scope}[xscale=-1]
  \draw[gray, dashed, xshift=-6cm] plot [smooth] coordinates { (-1,0.6)(0,0.5) (3.5,1.8) (4.2,1.4)  (5.2, 1.8)(6,0.5) (7,0.6) }; 
\end{scope}
\draw[-,black, black](1.5,0.1)--(1.5,-0.1) node[below] {$\lambda_*$};
\draw[-,dashed, gray] (3.1,1.6)--(3.1,-0.1)node[below] {$\lambda_0$};
\node at (3,-2) {(c) Reaching a trouch at $\lambda_0+\pi$.};
\end{tikzpicture}
\caption{Exemplary illustrations for the method of moving planes. Here, $\lambda_*$ represents the reflection point due to the reflection criterion and $\lambda_0$ is as in \eqref{eq:lambda}. }\label{F:Alternatives}
\end{figure}
\end{center}

The first case can be excluded by the touching lemma. If on the other hand we reach a crest at $x=\lambda_0$ and $w_{\lambda_0}>0$ on $(\lambda_0,\lambda_0+\pi )$, then the boundary lemma  implies that
\[
	\phi^\prime(\lambda_0)>0,
\]
which is a contradiction to $\phi$ being continuously differentiable and having a crest at $x=\lambda_0$. If we reach a  trough at $\lambda_0+\pi$, then either $\phi$ touches $\bar \phi$ at  $\lambda_0+\pi$ or $w_{\lambda}$ changes sign on different sides  of  $\lambda_0+\pi$. The former can be excluded by the touching lemma while the latter can be dealt with by applying the boundary point lemma at $\lambda_0+\pi$ with corresponding adjustments in view of $w_{\lambda}\leq 0$ on $[\lambda_0+\pi, \lambda_0+2\pi]$.
We conclude that $\phi$ is symmetric. The fact that $\phi$ has exactly one crest per period follows essentially by the same argument. Repeating the method of moving plane for $\lambda\geq \lambda_*$ implies that there does not exist a crest in $[\lambda_*,0)$. To show that there does not exist a crest in $(-\pi,\lambda_*]$, we can apply the same method by moving $\lambda \leq \lambda_*$ towards $-\pi$ as long as $w_\lambda <0$ on  $(\lambda-\pi,\lambda)$. This process stops at or before the first trough in $(-\pi,\lambda_*]$ and the same argument as before yields a contradiction to the assumption that $\phi$ has a crest in $(-\pi,\lambda_*]$.
We deduce that $\phi$ is symmetric and has exactly one crest per period. By translation, we may assume that the crest is located at $x=0$.
 In particular, $\phi^\prime(x)\geq 0$ for all $x\in [-\pi,0]$. We are left to show that the strict inequality prevails for any $x\in (-\pi,0)$. Equation \eqref{eq:new} can be written as
\[
 B_h+\frac{c^2}{4}-\left(\frac{c}{2}-\phi \right)^2 =L\phi.
\]
Let $x\in (-\pi,0)$, then
\[
2\left(\frac{c}{2}-\phi \right) \phi^\prime(x)=-\left(L\phi\right)^\prime(x)=-\int_{-\pi}^0 \left[ K(x-y)-K(x+y) \right]\phi^\prime(y)\,dy.
\]
Here we used the symmetry of $K$ and $\phi$. In the same fashion as in the proof of the touching lemma (cf. Lemma \ref{touching lemma}), one can show that the right-hand side is strictly positive, unless $\phi$ is a trivial solution. But this is already excluded by our assumption on the wave profile.
\end{proof}

The proof above relies not only on the touching lemma, but also on Lemma \ref{boundary lemma}, which requires continuously differentiable solutions. If $\phi$ is a highest wave, that is $\max_{x\in\T} \phi(x)= \frac{c}{2}$, then the differentiability of $\phi$ is no longer guaranteed, see Proposition \ref{prop:C1}. However, assuming in addition to the refection criterion, that the highest wave $\phi$ has a \emph{unique} global maximum with height $\frac{c}{2}$ per period, we prove that $\phi$ is symmetric and has a monotone profile.

\begin{thm}[Symmetry of highest waves]\label{thm:S}
Let $\phi\leq \frac{c}{2}$ be a $2\pi $-periodic, bounded solution with $\max_{x\in \T} \phi =\frac{c}{2}$. Assume that $\phi$ has a unique global  maximum in $\T$ and satisfies the reflection criterion.
Then $\phi$ is symmetric and has exactly one crest per period. Moreover,
\[
\phi^\prime(x)>0\qquad \mbox{for all}\qquad x\in (-\pi,0),
\]
after a suitable translation.
\end{thm}

\begin{proof}
Suppose by contradiction that $\phi$ is not symmetric. After proper translation and reflection, we may assume that a global minimum is located at $x=-\pi$ and the unique global maximum at some point $x_1 \in [0,\pi)$ and  $\lambda_*\in (-\pi,0]$ (cf. Figure \ref{F:reflection}).
\begin{figure}[h!]
\begin{tikzpicture}[scale=0.55]
\small
\draw (-2,0)--(8,0);
\draw[-,black](0,0.1)--(0,-0.1) node[below] {$-\pi$};
\draw[-,black](6,0.1)--(6,-0.1) node[below] {$\pi$};
\draw[-,black](3,0.1)--(3,-0.1) node[below] {$0$};
\draw[luh-dark-blue, very thick] plot [smooth] coordinates { (-1,0.6)(0,0.5) (2.5,1.8)};
\draw[luh-dark-blue, very thick] plot [smooth] coordinates { (2.5,1.8)(3.2,1) (6,0.5) (7,0.6) } ;
\begin{scope}[xscale=-1]
  \draw[gray, dashed, xshift=-3.5cm] plot [smooth] coordinates {  (-2.8,0.9)(0,0.5) (2.5,1.8)};
  \draw[gray, dashed, xshift=-3.5cm] plot [smooth] coordinates { (2.5,1.8)(3.2,1) (5,0.6) } ;
\end{scope}
\draw[-,dashed] (1.75,1.4)--(1.75,-0.1)node[below] {$\lambda_*$};
\node[align=center] at (3,-2)  {\scriptsize{Wave profile with unique crest in $(-\pi,0]$}\\ \scriptsize{satisfying the reflection criterion with $\lambda_*$.}};
\draw[->] (10,1)--(13,1);
\node at (11.5,1.5) {\scriptsize {Reflection about $x=0$}};
\end{tikzpicture}
\qquad
\begin{tikzpicture}[scale=0.55]
\small
\draw (-2,0)--(8,0);
\draw[-,black](0,0.1)--(0,-0.1) node[below] {$-\pi$};
\draw[-,black](6,0.1)--(6,-0.1) node[below] {$\pi$};
\draw[-,black](3,0.1)--(3,-0.1) node[below] {$0$};
\draw[gray, dashed, xshift=-4cm] plot [smooth] coordinates { (3.2,1) (6,0.5) (8.5,1.8) } ;
\draw[gray, dashed, xshift=-4cm] plot [smooth] coordinates { (8.5,1.8)(9.2,1) (12,0.5)} ;
\begin{scope}[xscale=-1]
  \draw[orange, very thick,xshift=-6cm] plot [smooth] coordinates { (-1,0.6)(0,0.5) (2.5,1.8)};
  \draw[orange, very thick, xshift=-6cm] plot [smooth] coordinates{ (2.5,1.8)(3.2,1) (6,0.5) (7,0.6) } ;
\end{scope}
\draw[-,dashed] (1,0.6)--(1,-0.1)node[below] {$\bar \lambda_*$};
\node[align=center] at (3,-2)  {\scriptsize{Reflected wave profile with unique crest in $[0,\pi)$}\\ \scriptsize{satisfying the reflection criterion with $\bar \lambda_*$.}};
\end{tikzpicture}
\caption{By a possible refection about $x=0$, one can assume that the unique global maximum per period is located in $[0,\pi)$ and the reflection criterion stays valid.}\label{F:reflection}
\end{figure}
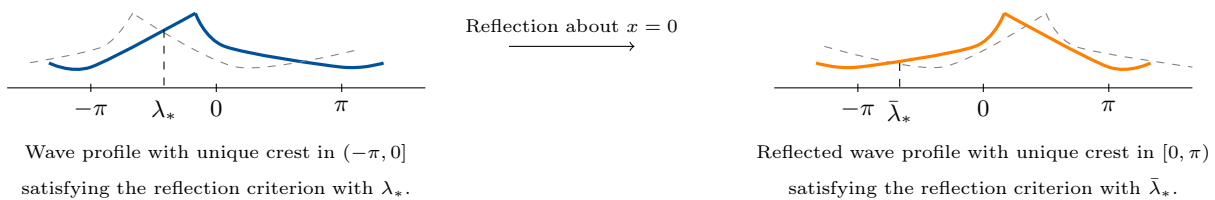

  Note that $\phi$ is smooth on $\R\setminus\{x_1+2\pi \Z\}$ due to Proposition~\ref{prop:reg}. As in the proof of Theorem \ref{thm:Snew}, let $w_\lambda:\R\rightarrow \R$ be the \emph{reflection function} around $\lambda$ given by
\[w_\lambda(x):= \phi(x)-\phi(2\lambda -x)\]
and recall that due to the symmetry of $K$ the function $\phi(2\lambda-\cdot)$ is a solution whenever $\phi$ is a solution. 
Again, set 
	\[\lambda_0:= \sup \{\lambda \in [\lambda_*,0 ] \mid w_\lambda(x) > 0 \; \mbox{ for all } \; x\in (\lambda,\lambda+\pi )\}.\]
The argument can be carried out in analog to the proof of Theorem \ref{thm:Snew}, since $\phi$ is smooth on $[-\pi,0)$.
Consider $w_\lambda$ for $\lambda\geq \lambda_*$. Starting at $x=\lambda_*$, we move the plane $\lambda$ about which the wave profile is reflected forward as long as $w_{\lambda}>0$ on $(\lambda,\lambda+\pi )$. Clearly, this process stops at or before the first crest in $[\lambda_*,0)$ at $\lambda=\lambda_0$ at $\lambda=\lambda_0$. In fact one of three occasions will occur: 
Either there exists $\bar x \in (\lambda_0,\lambda_0+\pi )$ such that $w_{\lambda_0}(\bar x)=0$; or we reach a crest at $x=\lambda_0<0$; or we reach a  trough at $\lambda_0+\pi $. The first case can be excluded by the touching lemma. If on the other hand we reach a crest at $x=\lambda_0<0 $ and $w_{\lambda_0}>0$ on $(\lambda_0,\lambda_0+\pi )$, then the boundary lemma  implies that
\[
	\phi^\prime(\lambda_0)>0,
\]
which is a contradiction to $\phi$ being continuously differentiable and having a crest at $x=\lambda_0<0 $ where $\phi(\lambda_0)<\frac{c}{2}$. If we reach a  trough at $\lambda_0+\pi$, then either $\phi$ touches $\bar \phi$ at  $\lambda_0+\pi$ or $w_{\lambda}$ changes sign on different sides  of  $\lambda_0+\pi$. The former can be excluded by the touching lemma while the latter can be dealt with by applying the boundary point lemma at $\lambda_0+\pi$ with a corresponding adjustment in view of  $w_{\lambda}\leq 0$ on $[\lambda_0+\pi,\lambda_0+2\pi]$. 
We conclude that $\phi$ is symmetric. By translation we can assume that $\phi$ is even. The fact that $\phi$ has a single crest per period and $\phi^\prime(x)> 0$ for all $x\in [-\pi,0)$ can be shown by the same argument as in the proof of Theorem \ref{thm:Snew}.  
\end{proof}

\bigskip

\subsection*{Acknowlegments}
The author G.B. gratefully acknowledges financial support by the Deutsche Forschungsgemeinschaft (DFG) through CRC 1173.

\bigskip

\bibliographystyle{siam}
\bibliography{Bib-Whitham}

\begin{thebibliography}{10}

\bibitem{Alk}
{\sc A.~D. Aleksandrov}, {\em Uniqueness theorems for surfaces in the large.
  {V}}, Amer. Math. Soc. Transl. (2), 21 (1962), pp.~412--416.

\bibitem{AB}
{\sc W.~Arendt and S.~Bu}, {\em Operator-valued {F}ourier multipliers on
  periodic {B}esov spaces and applications}, Proc. Edinb. Math. Soc. (2), 47
  (2004), pp.~15--33.

\bibitem{A}
{\sc M.~N. Arnesen}, {\em Decay of solitray waves}, arXiv:1906.03407,  (2019).

\bibitem{Boa}
{\sc R.~P. Boas, Jr.}, {\em Integrability theorems for trigonometric
  transforms}, Ergebnisse der Mathematik und ihrer Grenzgebiete, Band 38,
  Springer-Verlag New York Inc., New York, 1967.

\bibitem{BD}
{\sc G.~Bruell and R.~N. Dhara}, {\em Waves of maximal height for a class of
  nonlocal equations with homogeneous symbols}, to appear in Indiana Univ.
  Math. J.,  (2019).

\bibitem{BEGP}
{\sc G.~Bruell, M.~Ehrnstr\"{o}m, A.~Geyer, and L.~Pei}, {\em Symmetric
  solutions of evolutionary partial differential equations}, Nonlinearity, 30
  (2017), pp.~3932--3950.

\bibitem{BEP16}
{\sc G.~Bruell, M.~Ehrnstr\"{o}m, and L.~Pei}, {\em Symmetry and decay of
  traveling wave solutions to the {W}hitham equation}, J. Differential
  Equations, 262 (2017), pp.~4232--4254.

\bibitem{Constantin-Escher2004b}
{\sc A.~Constantin and J.~Escher}, {\em Symmetry of steady deep-water waves
  with vorticity}, European J. Appl. Math., 15 (2004), pp.~755--768.

\bibitem{Constantin-Escher2004}
\leavevmode\vrule height 2pt depth -1.6pt width 23pt, {\em Symmetry of steady
  periodic surface water waves with vorticity}, J. Fluid Mech., 498 (2004),
  pp.~171--181.

\bibitem{Constantin-Strauss2002}
{\sc A.~Constantin and W.~Strauss}, {\em Exact periodic traveling water waves
  with vorticity}, C. R. Math. Acad. Sci. Paris, 335 (2002), pp.~797--800.

\bibitem{Constantin-Strauss2004}
\leavevmode\vrule height 2pt depth -1.6pt width 23pt, {\em Exact steady
  periodic water waves with vorticity}, Comm. Pure Appl. Math., 57 (2004),
  pp.~481--527.

\bibitem{Craig}
{\sc W.~Craig and P.~Sternberg}, {\em Symmetry of solitary waves}, Comm.
  Partial Differential Equations, 13 (1988), pp.~603--633.

\bibitem{D}
{\sc J.~Dubourdieu}, {\em Sur un th{\'e}or{\`e}me de {M}. {S}. {B}ernstein
  relatif {\`a} la transformation de {L}aplace-{S}tieltjes}, Compositio Math.,
  7 (1940), pp.~96--111.

\bibitem{Dubreil-Jacotin}
{\sc M.~L. Dubreil-Jacotin}, {\em Sur la d{\'e}termination rigoreuse des ondes
  permanents periodiques d'ampleur finie.}, J. Math. Pure. Appl., 13 (1934),
  pp.~217--291.

\bibitem{EEW}
{\sc M.~Ehrnstr\"{o}m, J.~Escher, and E.~Wahl\'{e}n}, {\em Steady water waves
  with multiple critical layers}, SIAM J. Math. Anal., 43 (2011),
  pp.~1436--1456.

\bibitem{Ehrn-Holden-Ray}
{\sc M.~Ehrnstr\"{o}m, H.~Holden, and X.~Raynaud}, {\em Symmetric waves are
  traveling waves}, Int. Math. Res. Not. IMRN,  (2009), pp.~4578--4596.

\bibitem{EK}
{\sc M.~Ehrnstr{\"o}m and H.~Kalisch}, {\em Traveling waves for the {W}hitham
  equation}, Differential Integral Equations, 22 (2009), pp.~1193--1210.

\bibitem{EW}
{\sc M.~Ehrnstr{\"o}m and E.~Wahl{\'e}n}, {\em On {W}hitham's conjecture of a
  highest cusped wave for a nonlocal dispersive equation}, Annales de
  l'Institut Henri Poincare. Analyse non lin\'{e}ar. vol. 36 (6).,  (2019).

\bibitem{Garabedian}
{\sc P.~R. Garabedian}, {\em Surface waves of finite depth}, J. Analyse Math.,
  14 (1965), pp.~161--169.

\bibitem{Gerstner1809}
{\sc F.~Gerstner}, {\em Theorie der {W}ellen samt einer daraus abgeleiteten
  {T}heorie der {D}eichprofile}, J. Fluid Mech., 2 (1809), pp.~412--445.

\bibitem{GNN}
{\sc B.~Gidas, W.~M. Ni, and L.~Nirenberg}, {\em Symmetry and related
  properties via the maximum principle}, Comm. Math. Phys., 68 (1979),
  pp.~209--243.

\bibitem{G}
{\sc L.~Grafakos}, {\em Classical {F}ourier analysis}, vol.~249 of Graduate
  Texts in Mathematics, Springer, New York, third~ed., 2014.

\bibitem{Guo}
{\sc S.~Guo}, {\em Some properties of completely monotonic sequences and
  related interpolation}, Appl. Math. Comput., 219 (2013), pp.~4958--4962.

\bibitem{Hur2001}
{\sc V.~M. Hur}, {\em Global bifurcation theory of deep-water waves with
  vorticity}, SIAM J. Math. Anal., 37 (2006), pp.~1482--1521.

\bibitem{Hur2007}
\leavevmode\vrule height 2pt depth -1.6pt width 23pt, {\em Symmetry of steady
  periodic water waves with vorticity}, Philos. Trans. R. Soc. Lond. Ser. A
  Math. Phys. Eng. Sci., 365 (2007), pp.~2203--2214.

\bibitem{Hs}
\leavevmode\vrule height 2pt depth -1.6pt width 23pt, {\em Symmetry of solitary
  water waves with vorticity}, Math. Res. Lett., 15 (2008), pp.~491--509.

\bibitem{Matioc2013}
{\sc A.-V. Matioc and B.-V. Matioc}, {\em On the symmetry of periodic gravity
  water waves with vorticity}, Differential Integral Equations, 26 (2013),
  pp.~129--140.

\bibitem{OS}
{\sc H.~Okamoto and M.~Sh{\=o}ji}, {\em The mathematical theory of permanent
  progressive water-waves}, vol.~20 of Advanced Series in Nonlinear Dynamics,
  World Scientific Publishing Co., Inc., River Edge, NJ, 2001.

\bibitem{Pei-DP}
{\sc L.~Pei}, {\em Exponential decay and symmetry of solitary waves to
  {D}egasperis-{P}rocesi equation}, J. Differential Equations, 269 (2020),
  pp.~7730--7749.

\bibitem{RT}
{\sc M.~Ruzhansky and V.~Turunen}, {\em On the {F}ourier analysis of operators
  on the torus}, in Modern trends in pseudo-differential operators, vol.~172 of
  Oper. Theory Adv. Appl., Birkh\"{a}user, Basel, 2007, pp.~87--105.

\bibitem{Serr}
{\sc J.~Serrin}, {\em A symmetry problem in potential theory}, Arch. Rational
  Mech. Anal., 43 (1971), pp.~304--318.

\bibitem{T3}
{\sc M.~E. Taylor}, {\em Partial differential equations {III}. {N}onlinear
  equations}, vol.~117 of Applied Mathematical Sciences, Springer, New York,
  second~ed., 2011.

\bibitem{Toland1998}
{\sc J.~F. Toland}, {\em On the symmetry theory for {S}tokes waves of finite
  and infinite depth}, in Trends in applications of mathematics to mechanics
  ({N}ice, 1998), vol.~106 of Chapman \& Hall/CRC Monogr. Surv. Pure Appl.
  Math., Chapman \& Hall/CRC, Boca Raton, FL, 2000, pp.~207--217.

\bibitem{Widder}
{\sc D.~V. Widder}, {\em The {L}aplace {T}ransform}, Princeton Mathematical
  Series, v. 6, Princeton University Press, Princeton, N. J., 1941.

\end{thebibliography}

	\end{document}